\documentclass{birkmult}

\newcommand{\beq}{\begin{equation}}
\newcommand{\eeq}{\end{equation}}

\newcommand{\beqn}{\begin{eqnarray}}
\newcommand{\eeqn}{\end{eqnarray}}

\newcommand{\RR}{{\rm I\hspace{-0.50ex}R} }
\newcommand{\CC}{\rm \hbox{C\kern-.57em\raise.47ex
         \hbox{$\scriptscriptstyle |$}\kern+0.5 em }}
\newcommand{\NN}{{\rm I\hspace{-0.50ex}N} }

\newcommand{\UU}{{\mathcal U} }

\newcommand{\reachable}[3]{ {\mathcal R}^{#1}_{#2}(#3)}

\newcommand{\closurereachable}{\overline{ {\mathcal R}_{\UU}(e)}}

\newcommand{\e}{\epsilon}

\newcommand{\p}{\Psi}

\newcommand{\idt}{ i \frac{\partial}{\partial t}}

 \newtheorem{thm}{Theorem}[section]
 
 \newtheorem{lem}[thm]{Lemma}
 
 \theoremstyle{definition}
 \newtheorem{defn}[thm]{Definition}
 \theoremstyle{remark}
 \newtheorem{rem}[thm]{Remark}
 
 \numberwithin{equation}{section}

\begin{document}
%
%
%
%
%
%
%
%
%
\title[Beyond bilinear quantum control]
{Beyond bilinear controllability : \\ applications to quantum control}
\author[Gabriel Turinici]{Gabriel Turinici}

\address{%
CEREMADE, Universit\'e Paris Dauphine \\
Place du Mar\'echal De Lattre De Tassigny\\
75775 PARIS CEDEX 16\\
FRANCE}

\email{Gabriel.Turinici@dauphine.fr}

\thanks{This work was partially supported by INRIA-Rocquencourt and 
CERMICS-ENPC, Champs sur Marne, France. The author acknowledges
an ACI-NIM grant from the Minist\`ere de la Recherche, France.}


\date{October 1st, 2005}

\begin{abstract}
Quantum control is traditionally expressed through bilinear models and their 
associated Lie algebra controllability criteria. But, the 
first order approximation are not always sufficient and
higher order developpements are used in recent works. Motivated
by these applications, we give in this paper a criterion that applies
to situations where the evolution operator is expressed as sum of
possibly non-linear 
real functionals of the {\bf same} control that multiplies some
time independent (coupling) operators.
\end{abstract}

\maketitle
\tableofcontents
\section{Background on quantum control}

Controlling the evolution of molecular systems at quantum level
has been envisioned from the very beginings of the laser technology.
However, approaches based on designing laser pulses based on
intuition alone did not succed in general situations due to
the very complex interactions that are at work between the laser
and the molecules to be controlled, which results e.g., in the
redistribution of the incoming laser energy to the whole molecule.
 Even if this
circumstance initially slowed down investigations in this area,
the realization
that this inconvenient can be recast and attacked with the tools of
(optimal) control theory~\cite{hbref110} greatly contributed to the
first positive experimental
results~\cite{hbref1,hbref3,hbref4,hbref5,hbref6,hbref112,hbref113}.

The regime that is relevant for this work is related to time scales of the
order of the femtosecond
($10^{-15}$) up to picoseconds ($10^{-12}$) and the space scales
from the size of one or two atoms to large polyatomic molecules.

Historically, the first applications that were envisionned were the
manipulation of chemical bonds (e.g., selective dissociation) or isotopic
separation. Although initially, only few atoms molecules were investigated
(di-atomics) the experiments soon were designed to treat more complex
situations~\cite{hbref1} as selective bond dissociation in an organi-metalic
complex $CpFe(CO)_2Cl$ ($Cp$ is the cyclopentadienyl ion)
by maximizing or minimizing the quotient
of $CpFeCOCl^+$ ions obtained with respect to $FeCl^+$ ions.

Continuing this breakthrough, 
other poly-atomic molecules were
considered in strong fields. For instance, in ~\cite{hbref3} the molecules are
the acetone $(CH_3)_2 CO$,  the trifluoroacetone $CH_3\-COCF_3$
and the acetophenone $C_6 H_5 COCH_3$. Using tailored laser pulses
it was shown possible to obtain $CH_3CO$ from
$(CH_3)_2 CO$, $CF_3$ (or $CH_3$) from  $CH_3 COCF_3$
but also $C_6H_5CH_3$ (toluene) from  $C_6 H_5 COCH_3$.

But the applications of laser control do not stop here. High
Harmonic Generation)~\cite{hbref7} is a technique that allows to obtain
output lasers whose frequency is large interger multiples of the
input pulses.

A different class of applications works in a different regime of
shorter time scales and large intensity. This regime is additionally not
compatible with the standard  Born-Oppenheimer approximation
and requires to consider both nucleari and electrons as quantum particles with
entangled wevefunction~\cite{handbookbandrauk}.

In a different framework,
the manipulation of quantum states of atoms and molecules allows
to envision the construction of quantum computers~\cite{Deu85a,Sho94a}

Finally, biologically related applications are also the object of
ongoing research.

\section{Background on controllability criteria}

We start in this section to investigate the theoretical controllability
results that are nowadays available for quantum systems. The evolution
of the system will be
described by the driving
Schr\"odinger equation (we work here in atomic units i.e.
$\hbar=1$)
\beqn \label{eq:Schrt}
& \ & \idt {\p(t,x)} =  H(t) {\p(t,x)}  \\
& \ & \nonumber \p(t_0,x) = \p_0(x).
\eeqn
where $H(t)$ is the Hamiltonian of the system and $x \in \RR^\gamma$ the set of
internal degrees of freedom.
We introduce the Hilbert space structure given by the scalar product
\beq
\langle f , g \rangle = \int_{\RR^\gamma} \overline{f(x)} g(x) dx
\eeq
where 
$\overline{a+ib} = a-ib$ the conjugate of a complex number.

We only consider in this paper situations when
the Hamiltonian is auto-adjoint
$H(t)^\dag = H(t)$; we denoted by $T^\dag$ the adjoint
of a operator $T$. 
The auto-adjointeness of $H$ implies that the
${L^2_x({\RR}^{{\gamma}})}$
 norm of the evolving state is conserved.
Indeed
\beqn
& \ & 
\frac{d}{dt} \|\p(x,t)\|_{L^2_x({\RR}^{{\gamma}})}=
\frac{d}{dt}  \langle \p(x,t),\p(x,t) \rangle 
\nonumber \\ & \ & 
= \langle \frac{d}{dt}  \p(x,t),\p(x,t) \rangle 
+ \langle \p(x,t) , \frac{d}{dt} \p(x,t) \rangle 
\nonumber \\ & \ & 
= \langle \frac{H(t)}{i}  \p(x,t),\p(x,t) \rangle 
+ \langle \p(x,t) , \frac{H(t)}{i} \p(x,t) \rangle  = 0.
\eeqn
Thus 
\begin{equation} \label{consv}
\|\p(x,t)\|_{L^2_x({\RR}^{{\gamma}})}=
\|\p_0\|_{L^2({\RR}^{{\gamma}})},
 \ \forall t>0,
\end{equation}
so the  wave function $\p(t)$, evolves
on the (complex) unit sphere 
$$S = \left\{ {\psi  \in L^2(\RR^\gamma)\,:
\,\left\| \psi  \right\|_{L^2 (\RR^\gamma)}  = 1} \right\}.$$

When the system  evolves freely under its own internal
dynamics i.e. when isolated molecules are considered, 
the free evolution Hamiltonian $H_0$ is introduced. This Hamiltonian
is the
sum of the kinetic part $T$ and the potential operator $V(x)$ : 
$H_0=T + V(x)$. A prototypical example of $T$ is the Laplace operator
while for $V(x)$ one can encounter Coulomb potential or 
Lennard-Jones type dependence.
We obtain the following evolution
in the absence of external interaction:
\beqn \label{eq:Schr0}
& \ & \idt {\p(t,x)} =  H_0 {\p(t,x)}  \\
& \ & \nonumber \p(t_0,x) = \p_0(x).
\eeqn
But, when the free evolution of the system does not
generate a satisfactory dynamical output, an external
interaction is introduced to {\it control} it. 
An example  of 
external control of paramount importance is a
laser source of intensity
$\e(t) \in \RR, \ t \ge 0$. 

The purpose of control may be formulated as
to drive the system from its initial state $\p_0$ to take a convenient
dynamical path to a final state compatible with
predefined requirements. The control is here the laser intensity $\e(t)$.
We will come back later with details on the laser field $\e(t)$.

This laser will modify the 
Hamiltonian $H(t)$ of the system. A first order approximation
can be considered by introducing a 
time-independent dipole moment operator $\mu(x)$ resulting
in the dynamics:
\beqn \label{eq:Schr}
& \ & \idt {\p(t,x)} = \left( H_0 + \e(t)\mu \right) {\p(t,x)}  \\
& \ & \nonumber \p(t_0,x) = \p_0(x).
\eeqn
This is the so-called {\it bi-linear} framework
(the control enters linearly multiplying the state),
  that is the object of 
much theoretical and numerical work in quantum control. We also review 
below some of the results that are available in this formulation.
However, recently, 
higher order field dependence has  been considered in
different circumstances see e.g., ~\cite{DION1,DION2} 
for details. In these situations the Hamiltonian $H(t)$ is
developped further as~:
\beq \label{eq:H_high}
H(t) = H_0 +\e(t)\mu_1  + \e^2(t) \mu_2 + ... + \e^L(t) \mu_L.
\eeq

The question that will be of interest to us in this work is the study of
all possible final states for the quantum system. This question is important
in order to understand the
capabilities that a laboratory experiment will be able to provide
and also, in a more general setting, to accompany
the introduction of new experimental protocols.

More specifically, we will show how the criteria available for bilinear
control can be extended to treat the Hamiltonian~\eqref{eq:H_high} where
a {\bf single} control amplitude $\e(t)$  appears before different
coupling operators $\mu_1,...,\mu_L$.

Many of the questions regarding the properties of the quantum control
procedures, such as controllability, optimal control definition, etc, ...
need, in order to be defined, to specify the admissible control
class, i.e., the set $\UU$ where the control $\e(t)$ is alowed
to vary. Among the properties that can define this admissible set,
some are related to the regularity of the time-dependence ($L^2$,
$H^1$, ... etc) or of the Fourier expression (sum of sinusoidal functions
multiplied by an overall enveloppe, etc,...) or to additional
structure: e.g. piecewise continuous, piecewise constant, locally
bounded ...

The choice of one or several conditions in the list above is
motivated in practice by capability to reproduce that
particular form or to inherent experimental restrictions
(finite total laser energy/fluence, etc). As the laser technology
is constantly evolving, the first class of constraints becomes
less critical and thus it is realistic to consider very weak
constraints on the control set, e.g. 
 $\UU = L^2(\RR) \cap L^\infty_{loc}(\RR)$.

However, to treat even more general situations,
we will consider in this work controls $\e(t)$ that are
piecewise constant, taking any value in a set $V$, which will remain
fully general.

\subsection{Infinite dimensional bilinear control}\label{sec:infinitedim}

When compared to the finite dimensional control equations (see 
Section~\ref{subsec:bilinearcontrol}),
controllability of the infinite dimensional version of the
bilinear Time Dependent
Schr\"odinger Equation is much less understood at this time. In fact,
most of the progress obtained so far takes the form of
 negative results, in contradiction with the positive results available
in finite dimensional settings. However we see
the absence of positive controllability results 
is rather a failure of
today's control theory tools to provide insight into controllability
rather than an actual restriction.
We do believe that new tools and concepts will make positive results
possible.

Let us write the solution of~\eqref{eq:Schr} in the following form:
\beqn
\p(t)=e^{-iH_0t}\p_0-i\int_0^t\e(s)e^{-iH_0(t-s)}\mu \p(s)ds
\eeqn
This formulation (see~\cite{cazenave} for details) is granted
by the properties of the operator
$\mu : H^1_0(\RR^\gamma) \to  H^{-1}(\RR^\gamma)$
which is continuous when $\mu$ is bounded; we also recall that the
control $\e$ can be considered bounded in both $L^\infty$ and $L^2$.

The application $\e(t) \mapsto \p(x,t)$
 possesses an important compacity property which is
the key of the controllability results 
(we refer the reader interested in details 
to~\cite{bms,thesejulien})~:
\begin{lem} \label{lemma:compacity_bilinear}
Suppose that $\mu : X \to X$
is a bounded operator and that $H_0$ generates a $C^0$
semigroup of bounded linear operators on some Banach space $X$
(e.g. $X = H^1_0(\RR^\gamma)$). Denote for $T>0$ and
$\e \in L^1([0,T])$ by $\p_\e(x,t)$ the
solution  of~\eqref{eq:Schr} with control $\e$.
Then  $\e \mapsto \p_\e$ is a compact mapping in the sense that for any
$\e_n$ that converges weakly to $\e$ in $L^1([0,T])$ $\p_{\e_n}$ converges
strongly in $C([0,T]; X)$ to $\p_\e$.
\end{lem}

This compactness property allows to give 
{\it negative results} for general
bilinear controllability settings as in~\cite{bms} where they
 were applied
to the wave and rod equations. Specific statements
for quantum control have been latter derived
(Thm.~1 from \cite{hbref31} ; see also \cite{bms,gtsydney}) and can
be stated as:

\begin{thm}\label{negative}
 Let $S$ be the complex unit sphere of $L^2(\RR^\gamma)$.
Let ${\mu}$ be a bounded operator from the Sobolev
space $X$ (e.g., $X = H^1_x({\RR}^{{\gamma}})$)
to itself and let $H_0$ generate
a $C^0$ semigroup of bounded linear operators on $X$.
Denote by $ \p_\epsilon (x,t)$
the solution of
\eqref{eq:Schr}. Then
the set of attainable states from $\p_0$ defined by
\begin{equation}
\label{attain}
{\mathcal AS} =
\cup_{T > 0} \{ \p_\epsilon (x,T) ; \epsilon(t) \in L^2([0,T]) \}
\end{equation}
is contained in a countable union of compact subsets
of $X$. In
particular its
complement
$ S \cap X  \ \backslash \ {\mathcal AS} $
 with respect to $S \cap X$
is everywhere dense on  $S \cap X$.
The same holds true
for the complement with respect to
$S$.
\end{thm}

In a different formulation, the theorem implies that for any $\p_0  \in
X  \cap S$,
within any open set around an arbitrary point $\p  \in
X  \cap S$
there exists a state unreachable from $\p_0$ with $L^2$ controls.

\begin{rem}
Note that the result does not give information on the
closure of the set ${\mathcal AS}$. In particular it may well be
that while  ${\mathcal AS}$ still has dense complement its closure be
the whole space $X$. This would be the so-called
approximate controllability i.e. the possibility to reach targets
arbitrarily close to any given final state.
Despite some attempts in the litterature, at this time
there is no answer (positive or negative) to this question.
Among the ingredients that make this study
difficult
we can mention the possibility to use arbitrary large final time $T$,
the necessity to treat the continuous spectrum of the operator $H_0$
and the intrinsically unbounded  domain 
on which the problem is posed.
\end{rem}

To complicate even more the landscape, 
situations exists where
 the results obtained in infinite and finite dimensional
representation are of different nature. We will illustrate with a 
classical result on the harmonic oscillator.

\begin{lem} \label{cor:harmonic_infinitedim}
The infinite dimensional
harmonic oscillator $H_0 = -\frac{\partial^2}{\partial x^2} + x^2$,
$\mu=x$ is not controllable. Moreover the set of all
admissible states is a low-dimensional manifold of $L^2$.
\end{lem}
\begin{proof}
Let us begin by noting that the operators
$-i H_0$ and $-i \mu$ form a Lie algebra of  dimension $4$. Indeed,
let us compute the iterated commutators  of
$H_0 = -\frac{\partial^2}{\partial x^2} + x^2$ and  $\mu = x$~:
\beqn
& \ &
[ i (-\frac{\partial^2}{\partial x^2} + x^2), i x ] =
2 \frac{\partial}{\partial x} \\
& \ & [ ix, \frac{\partial}{\partial x} ] = -i \\
& \ &
[ i (-\frac{\partial^2}{\partial x^2} + x^2), i \frac{\partial}{\partial x}]
 = -2i x
\eeqn
Thus the dimension of the Lie algebra = $4$ and as such the system cannot
be controllable (all the states are on a low dimensional
manifold of $L^2$).
We refer to~\cite{mazyar_rouchon_oscillator} for recent 
contributions 
when the algebra of the operators $H_0$ and $\mu$ is finite
dimensional.
\end{proof}

This result is to be contrasted with additional
works that show that any (spectral) truncation of the
harmonic oscillator is controllable (see~\cite{sonia1} for details).

What can be deduced from the above result is that truncating an infinite
dimensional system is not always justified and care must be taken to check
that the control obtained in the resulting finite dimensional approximation
remain a good control for the initial, infinite dimensional system.
Of course, this is not needed for situations which are
inherently finite dimensional quantum systems (e.g.,  spins).

\subsection{Finite dimensional bilinear control} 
\label{subsec:bilinearcontrol}

Here our focus will be on finite dimensional systems. 
We introduce  an orthonormal basis
$D = \{ \psi_i(x) ; i=1,..,N\} $
for a
finite dimensional space.
An important example of such a space is
the one spanned by the first $N$
eigenstates of the internal Hamiltonian $H_0$. This example is 
also motivated in bi-linear settings by the ``perturbation''
argument that considers the control term $\e(t) \mu$ as
a first order developpement of $H(t)$ around $H_0$. Note
however that no concept of ``smallness'' is introduced in the
definition of admissible controls $\UU$.

Denote by $M$ the linear space that $D$ generates,
and let
$H_{0;a,b} = \langle H_0 \psi_a, \psi_b \rangle $
and
$\mu_{\ell;a,b} = \langle \psi_a , \mu_\ell \psi_b \rangle $
be the expressions of the operators
$H_0$ and $\mu_\ell$  with respect to this basis, 
$\ell=1,...,L$. To keep notations simple we will still denote from now
on by $H_0$ and $\mu_\ell$ the resulting 
$N \times N$ symmetric matrices.

In the Galerkin approach, expressing the 
Schr\"odinger equation in the space $M$ is equivalent to
supposing $\p(x,t) = \sum_{i=1}^N  \psi_i(x) c_i(t)$.
\beqn \label{eq:finitebilinear}
& \ & i \frac{d c(t;\e;c_0)}{dt} = H_0 c(t;\e;c_0) 
+ \left[ \e(t)\mu_1 + ...+ \e^L(t) \mu_L \right] c(t;\e;c_0)   \\
& \ & \nonumber c(t=0;\e;c_0) = c_0.
\eeqn
In the following, when no ambiguity prevents it,
 we will also simply denote
$c= c(t;\e;c_0)$. The finite dimensional counterpart of the 
norm conservation property~\eqref{consv}  reads:
$$ \sum_{n=1}^N |c_n|^2 = 1.$$
i.e., the state $c$ evolves on the unit sphere $S_N$ of
$\CC^N$.
The controllability can be formulated in this case as:
\begin{defn}
The system ($H_0,\mu_1,...,\mu_L$) is called (wavefunction) controllable, if
for any two states $c_k \in S_N$, $k=1,2$ there exists a final time
$T< \infty$ and control
$\e(t) \in L^2([0,T])$ such that the solution of
eqn.~\eqref{eq:finitebilinear} starting from $c_1$ ends in
$c_2$ at final time $T$: $c(T;\e;c_1)=c_2$.
\end{defn}

Although specific results for this setting exist~\cite{controllability0,controllability1}, a different 
alternative is to 
see~\eqref{eq:finitebilinear} as a system posed on  $U(N)$\footnote{$U(N)$ is
the set of all $N\times N$ complex unitary matrices.}.
We introduce the evolution equation on $U(N)$:
\beqn \label{eq:finitebilinearUN}
& \ & i \frac{d U(t;\e)}{dt} = \left[ 
H_0 +\e(t)\mu_1 + ... + \e^L(t)\mu_L
\right] U(t;\e)   \\
& \ & \nonumber U(t=0;\e) = Id.
\eeqn
Since $H_0$ and $\mu_\ell$ are symmetric matrices,
$U(t;\e)$ will remain unitary for
all $t\ge 0$. It is classical to remark then that the evolution of
$c(t;\e;c_0)$ can be obtained from the evolution of
$U(t;\e)$ by
$$c(t;\e;c_0) = U(t;\e)c_0.$$ In particular it follows that 
if the set of all attainable matrices $U(t;\e)$ is at least
 $SU(N)$ then the system is controllable. This is almost a necessary
condition for controllability, a notable exception being the
circumstance when $N$ is even: in this case, if the set of all
attainable matrices  contains $Sp(N/2)$ then controllability
still holds. We refer to~\cite{brockettsphere,albertini1} for
more detailed information.

Let us just mention that different representations of the system 
include the 
density matrix formulation with time dependent density
matrix operator $\rho(t)$ satisfying
\beqn \label{eqn:densityevolution}
& \ &
i \frac{\partial}{\partial t} \rho(t;\e;\rho_0) =
[ H_0  + \e(t)\mu_1 + ... +\e^L(t)\mu_L, \rho(t;\e;\rho_0)]
\\ & \ &
\rho(t=0;\e;\rho_0) = \rho_0
\nonumber
\eeqn
Then one can show 
$\rho(t;\e;\rho_0) = U(t;\e) \rho_0 U^\dag(t;\e)$. 
Controllability in this case is the possibility to steer
any initial mixed state $\rho_0$ to any other state $\rho^f$ unitarily
equivalent to it\footnote{A $N\times N$ matrix $\rho_2$ is
said unitarily equivalent to a $N \times N$ matrix $\rho_1$ if
there exists $M \in U(N)$ such that $\rho_2 = M \rho_0 M^\dag$.}.

Note that the
density matrix controllability
is equivalent to requiring that the set
of all matrices attainable from identity be at least $SU(N)$.

At a general level, the evolution
equation~\eqref{eq:finitebilinearUN} can be
re-written as
\beqn \label{eqn:lieevolution}
& \ & \frac{dx(t;\e;x_0)}{dt} = (A +\e(t) B_1 + ... \e^L(t)B_L) x(t;\e;x_0)
\\
& \ & x(0) = x_0.
\eeqn
where $x(t;\e;x_0)$ belongs to a Lie group $G$~(
see~\cite{bookliegroups,liebourbaki13,liebourbaki46} for basic facts about
the lie groups) and
$A$,$B_1$,...,$B_K$ to its associated Lie algebra $L(G)$.
The equation above is to be taken in the usual sense (using
the exponential map) when e.g.,
$\e(t)$ is piecewise continuous/constant and in a weak sense
(integral form) for general $\e(t)$~(see e.g.~\cite{bms} for additional
details).
For the quantum control problem
$A = -iH_0$ and $B_\ell=-i\mu_\ell$, $G = U(N)$.

\begin{rem}
Everything that will be said in this and following sections applies
with trivial modifications 
to the situation of several laser fields. For notational
convenience we will only give here the results for a unique laser field.
\end{rem}

We will denote by  $L_{A,B_1,...,B_k} \subset L(G)$  the
Lie algebra spanned by $A$, $B_k$, $k=1,...,K$ and by
$e$ the unity of $G$.

Let us now consider the set of all
reachable states from an initial state $y$:
\beq
\reachable{t}{\UU}{y} = \{ x(t;\e;y)
 \textrm{ solution of~\eqref{eqn:lieevolution} }; \e \in \UU \}.
\eeq
It is immediate to see that
\beq
\reachable{t}{\UU}{y} = \reachable{t}{\UU}{e}y
\eeq
and thus, describing the set $\reachable{t}{\UU}{e}$
allows to completely describe all other reachable
sets. When the final time is not specified, we will
denote
\beq
\reachable{}{\UU}{y} = \cup_{t\ge 0} \reachable{t}{\UU}{y}.
\eeq
The central question is to characterize $\reachable{}{\UU}{e}$.
When the bi-linear setting is considered i.e. $L=1$ and we note
$B=B_1$, we have the following result~\cite{jurd,lobry1}:
\begin{thm} \label{thm:liecontrollability}
Consider the system~\eqref{eqn:lieevolution}
defined on a Lie group $G$ with associated Lie algebra $L(G)$ containing
$A$ and $B$. If $G$ is
compact and the Lie algebra $L_{A,B}$ generated by $A$ and $B$ is
the complete algebra $L(G)$ : $L_{A,B} = L(G)$
then the set  $\reachable{}{\UU}{e}$ of all states states from
the identity is the Lie group $G$. Moreover,
there exists $0< T < \infty$ such that
$\reachable{T'}{\UU}{e} = G$ for all $T' \ge T$.
\end{thm}
\noindent This gives, when applied to quantum control~\cite{hbref30}:
($L=1$, $\mu = \mu_1$):
\begin{thm} \label{vish}
 If the  Lie algebra  $L_{-iH_0,-i\mu}$
generated by $-iH_0$ and $-i\mu$ 
has dimension $N^2$  (as a vector space over the real numbers) then
the system~\eqref{eq:finitebilinearUN} is density matrix controllable.
Furthermore, if both $-iH_0$ and $-i\mu$ are traceless then a
sufficient condition for the density matrix (thus wavefunction)
controllability of quantum system
is that the Lie algebra $L_{-iH_0,-i\mu}$ has dimension $N^2-1$.
\end{thm}

Although the results above conveniently address the situation of a bi-linear
setting, we are not aware of any similar results for 
the general quantum control situations~\eqref{eq:H_high}. In particular,
we know by the result above that, if $u_1$,...,$u_L$
are {\it independent controls}, i.e.,
\beqn \label{eqn:lieevolution2}
& \ & \frac{dx(t;\e;x_0)}{dt} = (A +u_1(t) B_1 + ...+ u_L(t)B_L) x(t;\e;x_0)
\\
& \ & x(0) = x_0,
\eeqn
an equivalent condition for the controllability
of the above system on its compact Lie group $G$ is that 
$A$,$B_1$,...,$B_L$ generate the whole Lie algebra $L(G)$.
But, there is no obvious way to say what will happen when the
controls $u_\ell$ are not independent but
related by the condition 
$u_\ell = \e^\ell(t)$. This study is the purpose of the next section.

\section{Criteria for non linear operators} \label{sec:nonlingen}

In order to extend the controlability results above 
beyond bi-linear interaction Hamiltonians, we will introduce
in this section a more general setting: we will rewrite
the control equation~\eqref{eqn:lieevolution} as
\beqn \label{eqn:lieevolutionH}
& \ & \frac{dx(t;\e;x_0)}{dt} = (F_1(\e(t)) B_1 + ... +
F_l(\e(t))B_L) x(t;\e;x_0)
\\
& \ & x(0) = x_0.
\eeqn
\noindent where $F_k : V \to \RR$ are real functionals.
Note in particular that we do not impose {\bf any} 
assumption on the regularity of the
functionals $F_k$.
Of course, one can recover the 
equation~\eqref{eqn:lieevolution} by setting $F_k(x) =x^k$ and adding
$F_0 = 1$.

In order to avoid trivialities, we will suppose in the following that
\beq
\textrm{ the functionals } (F_k)_{k=1}^L \textrm{ are linearly independent.} 
\eeq
Otherwise one may just consider a subset of functionals that are linearly
independent and adjust the matrices $B_k$ accordingly. Of course, since
we do not specify 
the set $V$ that lists all the possible control values
$\e$ the hypothesis above needs to be understood in the following 
acception: the functionals $F_k$ are said to be linearly
dependent if there exist constants $\lambda_1,...,\lambda_L \in \RR$ 
such that $\sum_{j=1}^L \lambda_j F_j(v) = 0$ for all $v\in V$. Otherwise
the functionals are said to be linearly independent.

In order to obtain the quantum controllability results, we begin
in this section with a controllability criterion 
on compact Lie groups. These results build on classical references
for bilinear controllability~\cite{jurd}.
We give first a weak but intuitive form and then we state
 the fully general one.

\begin{thm} \label{thm:approx}
Let~\eqref{eqn:lieevolutionH} be a control system posed on a compact
connected 
Lie group $G$, with linearly independent functionals $(F_k)_{k=1}^L$. 
Then 
if the  Lie algebra 
generated by $B_1$,...,$B_L$ is the full Lie algebra 
$L(G)$ of the group $G$, then 
 the system is approximatelly controllable, i.e. for any 
$a,b\in G$, $b$ is an accumulation point of the set of 
all states $x(t)$ attainable from $x(0) = a$ with admissible controls.
\end{thm}

\begin{proof} Let us begin by noting that if $F_k$ are independent then
there exist values $e_j \in V$, $j=1,...,L$ such that 
the vectors $v(e_j)=(F_1(e_j),...,F_L(e_j))$ are linearly independent. Suppose on the contrary that this is not true. 
Consider then a maximal set of vectors
$v(E_1),...,v(E_p)$ that are linearly independent. The matrix 
$(F_k(E_j))_{k=1;j=1}^{L;p}$ has rank preciselly $p$ and thus one
can extract $p$ functionals, denoted for notational convenience
$F_1,...,F_p$ such that $rank(F_k(E_j))_{k=1;j=1}^{p;p} = p$. Take now
some functional $F_{p+1}$ not in this set. It follows that
 $rank(F_k(E_j))_{k=1;j=1}^{p+1;p+1} = p$ 
and as such $det(F_k(E_j))_{k=1;j=1}^{p+1;p+1} = 0$ for any
$E_{p+1} \in V$. This determinant
can be computed as:
\beq \label{eq:lindep}
det(F_k(E_j))_{k=1;j=1}^{p+1;p+1} =
\lambda_1 F_1(E_{p+1}) + ... + \lambda_{p+1} F_{p+1}(E_{p+1}) = 0.
\eeq
Note that $\lambda_k$ do not depend on $E_{p+1}$
and that in particular 
$$\lambda_{p+1} = det (F_k(E_j))_{k=1;j=1}^{p;p}\neq 0.$$ 
Thus
eq.~\eqref{eq:lindep} implies that a linear combination with at least one
non-null coefficient $\lambda_{p+1}$ exists such that 
$\sum_{k=1}^{p+1} \lambda_k F_k(E) = 0 $ for all  $E \in V$. 
This is prevented by hypothesis. 

We have thus proved the existence of 
 $e_j \in V$, $j=1,...,L$ with 
the $v(e_j)=(F_1(e_j),...,F_L(e_j))$  linearly independent. This
means that $M_j = \sum_{k=1}^L F_k(e_j) B_k$ are also linearly
independent and span the same linear space as $B_k$, $k=1,...,L$
and thus $M_j$ span also the Lie algebra $L(G)$.
Moreover, all states $\{ e^{t M_j}x(0); t \in \RR_+ , j \le L\}$ 
are atainable from $x(0)$ for the control 
system~\eqref{eqn:lieevolutionH}.

It is clear that to prove approximate controllability
 is sufficient to set $a=e$ the neutral element of the group $G$. 
i.e. we have to prove that 
the closure $\closurereachable$  (with respect to the Lie group
topology)
of the reachable states from identity
is the whole $G$. From the 
hypothesis and sujectivity of the exponential mapping 
this is equivalent to proving that 
$$ \{e^M; M \in L(G) \} \subset \closurereachable.$$

We will begin by noting that $\closurereachable$ is a group. Indeed,
take two elements $x(t_1;\e_1;e),x(t_2;\e_2,e) \in \reachable{}{\UU}{e}$.
Then,defining the control $\e_{12}:[0,t_1+t_2] \to \RR$
by $\e_{12}(t) =\e_1(t)$ for all $0\le t \le t_1$ and
$\e_{12}(t_1+t) =\e_2(t)$ for all $0\le t \le t_2$ we obtain
$x(t_1+t_2;\e;e)= x(t_2;\e_2;e)x(t_1;\e_1,e)$ and thus
$ x(t_2;\e_2;e)x(t_1;\e_1,e) \in \reachable{}{\UU}{e}$. Hence
$\reachable{}{\UU}{e}$ is a semi-group which implies that
$\closurereachable$ is a semi-group too.

Let us now consider
$a \in \closurereachable$. Then $a^n \in \closurereachable$ for any
$n=1,2,...$. Since
$ \closurereachable \subset G$ which is a compact group,
$\closurereachable$ is compact at its turn. Then
there exists a sequence, that we can take such
that $n_k$ with $n_k - n_{k-1} \ge 2$,
with $a^{n_k} \to b \in G$. But then
$\closurereachable \ni a^{n_k - n_{k-1} -1 } \to bb^{-1} a^{-1}$
and thus $a^{-1} \in \closurereachable$.

It is immediate to see that, since the
solution for the control $\e(t) \equiv e_j$ 
is $x(t;0,\e) = e^{tM_j}$ we have the
inclusion
$ \{ e^{t M_j} ; t \ge 0 , j \le L\} \subset \closurereachable$. Since
$\closurereachable$ is a group, we will also have
$ \{ e^{t M_j} ; t \in \RR; j\le L \} \subset \closurereachable$.
Consider now two matrices $X_1, X_2 \in L(G)$ such that
$$ \{ e^{tX_i} ; t \ge 0 \} \subset \closurereachable, \ i=1,2.$$
We invoke now the formula
\beq \label{eq:trotter}
e^{t [ X_1,X_2]} = \lim_{n\to \infty} \left(
e^{-tX_2/\sqrt{n}} e^{-tX_1/\sqrt{n}} e^{tX_2/\sqrt{n}}
e^{tX_1/\sqrt{n}}
\right)^n
\eeq
to conclude that
$$ \{ e^{t[X_1,X_2]} ; t \in \RR \} \subset \closurereachable.$$

Similarly, we use the formula
$e^{t_1X_1 + t_2 X_2} = \lim_{n\to \infty}
\left(
e^{t_1 X_1 /n} e^{t_2 X_2 /n}
\right)^n
$
to conclude that
$$ \{ e^{t_1 X_1+ t_2 X_2} ; t_1, t_2 \in \RR \} \subset \closurereachable.$$

We have thus proved that the set
$\{ M\in L(G) ; e^{tM}  \in \closurereachable; \ \forall t\in \RR \}$
contains $M_j$, $j=1,...,L$, 
is closed to commutation and is a real vector space.
Thus it contains $L(G)$ hence the conclusion of the
theorem.
\end{proof}

The Theorem above has the advantage to be both intuitive and self-contained. 
However it only gives approximate controllability results, which are not the
strongest forms available. But, in order to obtain  exact controllability
more involved techniques are needed. In the litterature,
similar situations are treated by making use of the Chow 
theorem~\cite{chow} and of the bi-linear control 
techniques~\cite{jurd,sussman2}. The criterion can be stated as follows:

\begin{thm} \label{thm:general}
Let~\eqref{eqn:lieevolutionH} be a control system posed on a compact
connected 
Lie group $G$ with linearly independent functionals $F_k : V \to \RR$,
$k=1, ...,L$
and piecewise constant controls $\e$ taking any value in some set $V$. 
Then a necessary and sufficient condition for the exact controllability
is that
the  Lie algebra $L_{B_1,...,B_L}$
generated by $B_1$,...,$B_L$ be the full Lie algebra 
$L(G)$ of the group $G$.
\end{thm}
\begin{proof} We recall 
(see also end of Section~\ref{subsec:bilinearcontrol}) that the set of 
attainable states is included in the set of attainable states for the system 
\beqn \label{eqn:lieevolution22}
& \ & \frac{dx(t;\e;x_0)}{dt} = [u_1(t) B_1 + ... +u_L(t)B_L] x(t;\e;x_0)
\\
& \ & x(0) = e,
\eeqn
whose controllability is equivalent to ``$L_{B_1,...,B_L}=L(G)$''.
Thus $L_{B_1,...,B_L}=L(G)$ is a necessary condition for controlability.
To prove that is also sufficient, consider as in the proof of the 
Theorem~\ref{thm:approx}, the matrices $M_j = \sum_{k=1}^L F_k(e_j)B_k$,
$j=1,...,L$ that generate the same Lie algebra
$L_{B_1,...,B_L}$. We recall that  all 
states $\{ e^{t M_j}; t \in \RR_+ , j \le L\}$ and all finite products of 
such states
are atainable from the identity $e$. We invoke now a technique present
in the proof of Thm.~3.1 of~\cite{sussman2}: for any $P \in \NN$
and any multi-index $i=(i_1,...,i_r) \in \{ 1,...,L\}^r$  
denote by $A(i,T)$ the atainable states with the sequence of
operators $i$ and total time less than $T$ : 
$$A(i,T) = \left\{ \prod_{\ell=1}^r e^{t_\ell M_{i_\ell}} ; 
\sum_{\ell=1}^r |t_\ell| \le P, t_1,...,t_r \in \RR \right\}.$$
 We know by the Chow
theorem that the union of the sets $A(i,T)$ is the whole Lie group $G$. Also,
it is immediate that any $A(i,T)$ is image of a compact set thus compact.
It follows by the Baire category theorem that 
$A(i,P)$ has non-empty interior
at least for a couple ($i$,$P$). 
For such an $i=(i_1,...,i_m)$ we introduce
 the mapping $F: \RR^m \to G$ defined by
$t=(t_1,...,t_\ell) \mapsto F(t) = \prod_{\ell=1}^r e^{t_\ell M_{i_\ell}} $.
This mapping is analytic and its image is has nonempty interior. By the Sard 
theorem its differential $dF(t)$ has full rank (i.e. equals 
the dimension of the
tangent space $TG$ of $G$) at least at some point $t$
and thus in a neighborhood. But since $dF(t)$ depends analytically
on $t$ the set of points where the rank is full is dense in $\RR^m$ and
as such the rank is full for some $t$ with all components strictly 
positive. Using a local inverse mapping theorem it follows that 
the image $F({\mathcal T})$ has non-empty interior where ${\mathcal T}$ is an open subset of
$\RR_+^m$. But all points in $F({\mathcal T})$ are realisable with admissible 
controls and thus the set of reachable points 
 $\reachable{}{\UU}{e}$ contains an open subset $D$ of $G$. 

By the previous Theorem,  $\reachable{}{\UU}{e}$ 
 is a subgroup i.e. for any 
 $y \in \reachable{}{\UU}{e}$ the set  $Dy$ is also reachable.
Since in addition $\reachable{}{\UU}{e}$ is dense in $G$ 
it follows that $ \reachable{}{\UU}{e} = G$.
\end{proof}

\section{Applications to quantum control}

The purpose of this section is to instantiate the results obtained 
previously to the specific situation of the quantum control.
We will give two results, one for the  density matrix formalism and
the second for the wave function.

\subsection{Density matrix}

To consider the specific situation of the density matrix formalism,
we use the results of the Section~\ref{sec:nonlingen} for the Lie group
$U(N)$. We obtain a first 
\begin{thm}
Consider the system
\beqn \label{eqn:densityevolutionGEN}
& \ &
i \frac{\partial}{\partial t} \rho(t;\e;\rho_0) =
[ H_0  + F_1(\e(t))\mu_1 + ...+ F_L(\e(t))\mu_L, \rho(t;\e;\rho_0)]
\\ & \ &
\rho(t=0;\e;\rho_0) = \rho_0
\nonumber
\eeqn
and suppose that the family $\{ 1, F_1,...,F_L\}$ 
is linearly independent.

Then, when at least one matrix  $H_0$,$\mu_1$,...,$\mu_L$
has nonzero trace, the equation~\eqref{eqn:densityevolutionGEN}
is density matrix controllable if and only if the Lie algebra 
$L_{iH_0,i\mu_1,...,i\mu_L}$
spanned by
the matrices $iH_0$,$i\mu_1$,...,$i\mu_L$ is the Lie algebra $u(N)$
of all skew-hermitian matrices or equivalently 
$dim_{\RR} L_{iH_0,i\mu_1,...,i\mu_L} = N^2$.

Otherwise, when all matrices 
$H_0$,$\mu_1$,...,$\mu_L$ have zero trace,
a necessary and
sufficient condition for controllability is that 
$L_{iH_0,i\mu_1,...,i\mu_L} = su(N)$
or equivalently 
$dim_{\RR} L_{iH_0,i\mu_1,...,i\mu_L} = N^2-1$.
\end{thm}
\begin{proof}
The first part of the conclusion follows from
Theorem~\ref{thm:general}
for the Lie group $G=U(N)$.

When all matrices have zero trace one uses the same result
for  $G=SU(N)$ noting that if two matrices $\rho_1$ and
$\rho_2$ are unitarily equivalent $\rho_2 = M \rho_1 M^\dag$ then
there exists $\gamma \in R$ with $M_{su}= Me^{i\gamma} \in SU(N)$ 
and  $\rho_2 = M_{su} \rho_1 M^\dag_{su}$.
\end{proof}

\noindent
An algorithmic verification of the above theorem can be devised as follows:
\begin{enumerate}

\item 
Test whether the functions 
$\{ 1, F_1,...,F_L\}$  are linearly independent. If the answer is yes
go to next step, otherwise keep only a subset $F_{i_1}, ... , F_{i_p}$
with $\{ 1, F_{i_1},...,F_{i_p}\}$ linearly independent and
modify the $B_1$ ... ,$B_L$ accordingly. For notational convenience
we suppose all functionals are independent ($p=L$).

\item Construct the traceless matrices
$\widetilde{H_0} = H_0 - \frac{Tr(H_0)}{N}Id$,
$\widetilde{\mu_1} = \mu_1 - \frac{Tr(\mu_1)}{N}Id$,...,
$\widetilde{\mu_L} = \mu_1 - \frac{Tr(\mu_L)}{N}Id$.
Denote by ${\mathcal O} = \{i\widetilde{H_0}, i\widetilde{\mu_1},...,
i\widetilde{\mu_L} \}$ .

\item Write any element of  ${\mathcal O}$ as a column vector
and compute the rank $r = rank({\mathcal O})$ over the real numbers.

\item Construct all commutators ${\mathcal C}$ 
of matrices in 
${\mathcal O}$\footnote{Some optimizations are possible at this point
as only new commutators are generally needed to be computed. We do not
enter into details here.}
and test whether  $rank ({\mathcal O} \cup {\mathcal C}) = r$. If not,
set ${\mathcal O} := {\mathcal O} \cup {\mathcal C}$ and
return to previous step.

\item Test whether $r=N^2 -1$. If yes the system is controllable, if not
the controllability does not hold.
\end{enumerate}
Even more precise results can be derived for the situation in
Eqn~\ref{eqn:densityevolution}.
\begin{thm}
Consider the developpement of the interaction Hamiltonian
$H =  H_0  + \e(t)\mu_1 + ... +\e^L(t)\mu_L$ 
resulting in the following evolution equation
\beqn \label{eqn:densityevolutionPART}
& \ &
i \frac{\partial}{\partial t} \rho(t;\e;\rho_0) =
[ H_0  + \e(t)\mu_1 + ... +\e^L(t)\mu_L, \rho(t;\e;\rho_0)]
\\ & \ &
\rho(t=0;\e;\rho_0) = \rho_0
\nonumber
\eeqn

Then, when at least one matrix  $H_0$,$\mu_1$,...,$\mu_L$
has nonzero trace, the equation~\eqref{eqn:densityevolutionPART}
is density matrix controllable if and only if the Lie algebra 
$L_{iH_0,i\mu_1,...,i\mu_L}$
spanned by
the matrices $iH_0$,$i\mu_1$,...,$i\mu_L$ is the Lie algebra $u(N)$
of all skew-hermitian matrices or equivalently 
$dim_{\RR} L_{iH_0,i\mu_1,...,i\mu_L} = N^2$.

Otherwise, when all matrices 
$H_0$,$\mu_1$,...,$\mu_L$ have zero trace,
a necessary and
sufficient condition for controllability is that 
$L_{iH_0,i\mu_1,...,i\mu_L} = su(N)$
or equivalently 
$dim_{\RR} L_{iH_0,i\mu_1,...,i\mu_L} = N^2-1$.
\end{thm}

\subsection{Wave function}

To derive results for the wave function 
of the same nature as the two criterions above one has to analyse
the transitive subsets of $U(N)$. We recall that a subset $A \subset U(N)$ is
called transitive when for any two vectors $a,b$ on the unit sphere
of $\CC^{N}$ there exists a matrix $X \in A$ with $b = Xa$. 
For the situation of quantum control, such a study is available in the 
literature~\cite{albertini1}. To be able to state the corresponding
result for this specific situation here, we 
introduce 
the centralizer ${\mathcal C}_G z$ of an element $z\in G$ which 
is defined
as the set of all elements that commute with $z$:
$${\mathcal C}_G z = \{ x \in G : xz = zx\}.$$
We also define $P = i \cdot diag(1,0,...,0) \in U(N)$. 

\begin{thm}
Consider the system
\beqn \label{eq:finitebilinear2}
& \ & i \frac{d c(t;\e;c_0)}{dt} = 
 \left[ H_0 + F_1(\e(t))\mu_1 + ...
+ F_L(\e(t)) \mu_L \right] c(t;\e;c_0)   \\
& \ & \nonumber c(t=0;\e;c_0) = c_0.
\eeqn
with $\| c_0 \| =1$. Suppose 
that the family $\{ 1, F_1,...,F_L\}$ 
is linearly independent and
denote by $L_{iH_0,i\mu_1,...,i\mu_L}$
the Lie algebra spanned by
the matrices $iH_0$,$i\mu_1$,...,$i\mu_L$.

Then the 
equation~\eqref{eq:finitebilinear2}
is (wave function) controllable if and only if
$$dim_{\RR} L_{iH_0,i\mu_1,...,i\mu_L} - 
dim ( L_{iH_0,i\mu_1,...,i\mu_L} \cap
{\mathcal C}_G P ) = 2N -2.$$

In particular a sufficient condition for controllability is that 
$$dim_{\RR} L_{iH_0,i\mu_1,...,i\mu_L} = N^2.$$
\end{thm}
\begin{proof} The proof follows from arguments in~\cite{albertini1}.\end{proof}

The following procedure allows to implement the above criteria:
\begin{enumerate}
\item 
Test whether the functions 
$\{ 1, F_1,...,F_L\}$  are linearly independent. If the answer is yes
go to next step, otherwise keep only a subset $F_{i_1}, ... , F_{i_p}$
with $\{ 1, F_{i_1},...,F_{i_p}\}$ linearly independent and
modify the $B_1$ ... ,$B_L$ accordingly. For notational convenience
we suppose all functionals are independent i.e. $p=L$.

\item 
Denote ${\mathcal O} = \{i {H_0}, i {\mu_1},...,
i{\mu_L} \}$ .

\item Write any element of  ${\mathcal O}$ as a column vector
and compute the rank $r = rank({\mathcal O})$ over the real numbers.

\item Construct all commutators ${\mathcal C}$ 
of matrices in 
${\mathcal O}$\footnote{Here again, optimizations are possible.}
and test whether  $rank ({\mathcal O} \cup {\mathcal C}) = r$. If not,
set ${\mathcal O} := {\mathcal O} \cup {\mathcal C}$ and
return to previous step.

\item Extract from 
${\mathcal O}$ the matrices 
that commute with $P$ and compute the rank $d$ of this ensemble over $\RR$.
Test whether $r-d=2N -2$. If yes the system is controllable, if not
the controllability does not hold.
\end{enumerate}

\noindent 
We also obtain
\begin{thm}
Consider the developpement of the interaction Hamiltonian
$H =  H_0  + \e(t)\mu_1 + ... +\e^L(t)\mu_L$ 
resulting in the following evolution equation
\beqn \label{eq:finitebilinear22}
& \ & i \frac{d c(t;\e;c_0)}{dt} = 
 \left[ H_0 + \e(t)\mu_1 + ...
+ \e^L(t) \mu_L \right] c(t;\e;c_0)   \\
& \ & \nonumber c(t=0;\e;c_0) = c_0.
\eeqn
with $\| c_0 \| =1$. 
 Denote by $L_{iH_0,i\mu_1,...,i\mu_L}$
the Lie algebra spanned by
the matrices $iH_0$,$i\mu_1$,...,$i\mu_L$.

Then the 
equation~\eqref{eq:finitebilinear22}
is (wave function ) controllable if and only if
$$dim_{\RR} L_{iH_0,i\mu_1,...,i\mu_L} - 
dim ( L_{iH_0,i\mu_1,...,i\mu_L} \cap
{\mathcal C}_G P ) = 2N -2.$$

In particular a sufficient condition for controllability is that 
$$dim_{\RR} L_{iH_0,i\mu_1,...,i\mu_L} = N^2.$$
\end{thm}

\begin{rem}
All the above results basically state that controllability with linearly
independent functionals of a single control $\e$ is true whenever
the same equation, but with completelly independent controls, is 
controllable.
\end{rem}


\end{document}